\documentclass[11pt]{article} 

\usepackage[centertags]{amsmath}
\usepackage{amsfonts}
\usepackage{amssymb}
\usepackage{amsthm}
\usepackage{newlfont}

\newcommand{\comment}[1]{}

\newtheorem{theorem}{Theorem}
\newtheorem{lemma}{Lemma}

\newtheorem{proposition}{Proposition}

\newtheorem{definition}{Definition}

\theoremstyle{definition}
\newtheorem{rem}{Remark}
\theoremstyle{definition}
\newtheorem{example}{Example}

\begin{document}

\title{Stochastic Comparisons between hitting times for  Markov Chains and words' occurrences }
\author{
\small Emilio De Santis \\ \small Department of Mathematics \\
\small La Sapienza University of Rome \\ \small {\tt
desantis@mat.uniroma1.it} 
\and
\small Fabio Spizzichino \\ \small Department of Mathematics \\
\small La Sapienza University of Rome \\ \small {\tt
fabio.spizzichino@uniroma1.it} }

  \maketitle

\begin{abstract}
We develop some sufficient conditions for the stochastic ordering
between hitting times, in a fixed state, for two Markov chains. In
particular, we focus attention on the so called \emph{skip-free}
case. In the analysis of such a case, we develop a special type of
coupling. We also compare different types of relations between two,
non-necessarily skip-free, Markov chains on the same state space.
Such relations have a natural role in establishing the usual
 and the asymptotic stochastic ordering  between the
probability distributions of hitting times. Finally, we present some
discussions and  examples related with words' occurrences.
\end{abstract}

\noindent {\small \textbf{Key Words}: Skip-free Markov chains,
Coupling, Asymptotic Stochastic Ordering, Spectral-Gap,

Word Occurrences, Leading Numbers.}

\noindent 2010 Mathematics Subject Classification: Primary: 60E15,
60J10

\medskip

\section{Introduction} \label{introd}

We consider a Markov chain $\mathbf{X}\equiv\{X_{n}\}_{n=0,1,\ldots
}$ on the state space $E=E_{k}\equiv\{0,1,\ldots ,k\}$ or
$E=E_{\infty}\equiv\{0,1,\ldots \}$. We denote by $T_{h}$ the
stopping times
\begin{equation}\label{stop}
    T_{h}=\inf\{n\in\mathbb{N}:X_{n}\geq h\},\text{ \ \ }h=1,2,\ldots
,k.
\end{equation}
$T_{h}$ is thus the random time needed to reach or exceed the
\textit{level} $h$. In particular we will consider 
transition matrices $P=\left(  p_{i,j}\right) _{i,j\in E}$ that are
\textit{skip-free (on the right)}, i.e.
\begin{equation}\label{proprieta}
    p_{i,j}=0\text{ \ \ \ \ \ \ if }1\leq i+1<j.
\end{equation}
In such cases, if the Markov chain starts from the state zero, $T_h$
is the hitting time of the state $h$.

Throughout the paper, we will denote by $\Upsilon_{k}$ the class of
transition matrices on the state space $E_k$ satisfying
\eqref{proprieta}. We also say that the Markov chain is
\emph{skeep-free} if its transition matrix is in $\Upsilon_{k}$.
 Besides the theoretical interest, the analysis of the hitting times $T_{h}$ for this class of Markov
chains emerges in the applications of probability to different
fields such as reliability, networks, biology, and so on.
 The literature devoted to these topics is then very wide. See in
 particular \cite{Miclo,Fill1,Fill2,DiaMic,StePak97, Zh} and references cited therein, for results
 concerning the probability distributions of the hitting times
 $T_{h}$, under different assumptions.
Skip-free Markov chains are encountered, in a fairly direct way, in
the problem of first occurrences of words in random sequences of
letters from a finite alphabet.
This problem will be briefly recalled in the last section. A very
large literature has been devoted to such a field, also, in
different frameworks and from different points of view. Typically,
attention has been concentrated on different aspects of the exact
computation of $\mathbb{E}\left( T_{k}\right) $ or of the
probability distribution of $T_{k}$.

In this paper we rather consider, for pairs of  Markov chains
$\mathbf{X}$ and $\widetilde{\mathbf{X}}$ on the same state space
$E_{k}$, stochastic orderings between the corresponding hitting
times $T_{k}$ and $\widetilde{T}_{k}$. The idea of studying
stochastic ordering of hitting times was already considered in the
papers \cite{IG,FP07b,DiCR}. In particular, some of our results are
in the same spirit of the paper by Irle and Gani, i.e. \cite{IG},
who consider also the context of detection of words.

Different notions of stochastic orders might be considered for the
$\mathbb{N}$-valued random variables $T_{k}$ and $\widetilde{T}_{k}$
(see e.g. \cite{ShSh,Sto}); as natural ones in our context, we
consider the \textit{usual stochastic order } $T_{k}$ $\preceq_{st}$
$\widetilde{T}_{k}$ and the \emph{tail (or asymptotic) stochastic
order}, that is defined in terms of tail behavior of the
distributions. A rather detailed analysis of the stochastic tail
order has been offered in the recent work \cite{HaijunLi}.


In our results concerning the usual stochastic order, the assumption
that $\mathbf{X}$ and $\widetilde{\mathbf{X}}$ belong to
$\Upsilon_{k}$ will be specifically used. The proof of our results
in such direction will be based on a coupling method that takes
essentially into account the order structure of the state space
$E_{k}$. More precisely, on a same probability space, we construct
two Markov chains (sharing the laws of $\mathbf{X}$ and
$\widetilde{\mathbf{X}}$, respectively) in such a way that they are
\textquotedblleft coupled" only in some instants when they visit the
same states. When one of the two chains has a transition to  a state
\textquotedblleft higher" than the other one, it stops and waits for
the latter, which has an independent evolution in the meantime. A
similar approach had been also developed in \cite{FP05}.

For the results concerning the asymptotic stochastic order, we use
different methods of proofs. In such a frame, we introduce a
specific order relation between two stochastic matrices of the same
size. The circumstance that such a relation is maintained under
products will have a relevant role in our derivations. In this part
of the paper the condition of skip-free will not be necessary.

Our results may also be used to deal with  continuous-time Markov
chains. Processes in continuous time with a property analogous to
\eqref{proprieta} have been called \emph{free of positive skips}
(see \cite{Kei}).

 In the specific cases of word occurrences the probability distributions of $T_{h}$ may appear rather simple at a
first glance. In particular, similarly to the geometric ones, they
are completely determined by their expected values. However they
manifest several apparently paradoxical aspects (see in particular
\cite{CZ1, DSS}).

Several results, in the literature concerning waiting times to
words' occurrences, have been based on the notion of \textit{leading
number }associated to a word. Such an analysis can appear, in a
sense, alternative to the one based on Markov chains. We will point
out that the two different approaches can be usefully compared and
combined.

The structure of the paper is as follows. In Section~\ref{sec2}, we
present our results concerning the stochastic order between hitting
times for skip-free Markov chains. In Section \ref{sec:asintotico},
we start by analyzing the asymptotic stochastic order for hitting
times. We will also show that the obtained results can moreover be
applied to establishing the usual stochastic order. In that section,
we consider a more general setting, where the skip-free condition is
dropped.
In Section~\ref{section3}, we discuss and apply the results of
previous sections in the context of waiting times to words'
occurrences. We also point out the interest of combining the
approaches respectively based on Markov chains and leading numbers.

\section{Stochastic comparisons between hitting times for skip-free Markov chains} \label{sec2}


We consider two skip-free Markov chains  $\mathbf{X}$ and $
\widetilde{\mathbf{X}}$, with transition matrices   $ P=(
p_{i,j})_{i,j \in E}$, $ \widetilde{P}=( \widetilde{p}_{i,j})_{i,j
\in E}$, and initial distributions  $\pi_0 =(\pi_0 (i))_{i \in E}$,
$\widetilde \pi_0 =(\widetilde \pi_0 (i))_{i \in E}$, respectively.
We furthermore consider $T_h$ and $\widetilde T_h$ where $T_h$ is
defined for   $\mathbf{X}$ in \eqref{stop} and $\widetilde T_h$ is
the analogue for $ \widetilde{\mathbf{X}}$. For a transition  matrix
$ P=( p_{i,j} )_{i,j \in E_k } $, we will use the notation
$p^{(n)}_{i, \cdot} =( p^{(n)}_{i, 0}, \ldots, p^{(n)}_{i, k})  $,
where $ p^{(n)}_{i, j} $ denotes the transition probability from $i
$ to $j $ in $n$ steps. The vector $p^{(n)}_{i, \cdot}$ will be seen
as a probability distribution over $ E_k$.

As far as probability distributions of hitting times are concerned,
several results have been given in \cite{DiaMic, Fill2, Fill1,
Miclo}. In particular, such distributions have been studied in terms
of the eigenvalues of the transition matrix obtained by making the
\textquotedblleft highest"  state $k$ absorbing. In \cite{BrSh,
DiaMic, Fill2} it has been shown that the distribution of $T_k$ is a
convolution of geometric distributions when all the eigenvalues are
non-negative real numbers.

More than on exact probability distributions, our interest is
focused on sufficient conditions for stochastic comparisons between
$T_h$ and $\widetilde T_h$. The afore-mentioned distributional
results cannot be directly applied for our purposes, even though
they provide complete characterizations.

In this section we obtain  sufficient conditions for the usual
stochastic comparisons between $T_h$ and $\widetilde T_h$. For basic
definitions and properties about such a stochastic order see
\cite{SM,Sto}.

Our first result can be seen as an extension of Theorem 4.1 in
\cite{IG}. Actually, we obtain a same type of conclusion still under
easily-applicable conditions that are however weaker and more
flexible. The proof of our result is based on a coupling method, as
proposed in \cite{FP05}.


\begin{theorem}\label{vedere}
Let $ P= (p_{i,j}: i, j =0, \ldots , k ) $ and $ \widetilde{P}=
(\widetilde{p}_{i,j}: i, j =0, \ldots , k ) $ be two transition
matrices in $\Upsilon_{k}$. Assume that, for any $ i =0, \ldots ,
k-1$, there exists $m(i)$ such that
\begin{itemize}
    \item[i)] $i+m(i) \leq k$;
    \item[ii)] $p^{(m(i))}_{i, \cdot} \succeq_{st}
\widetilde{p}^{(m(i))}_{i, \cdot}$.
\end{itemize}
 Moreover suppose that the initial
measures are stochastically ordered $ \pi \succeq_{st}
\widetilde{\pi} $. Then
\begin{equation}\label{conclprop3}
T_k \preceq_{st} \widetilde{T}_k.
\end{equation}
\end{theorem}
We defer the proof of Theorem 1 after a couple of remarks.

\begin{rem}
 \label{finito}
  In Theorem 4.1 in \cite{IG}, the case where $m(i)=1$ for $i=1,\ldots ,k-1$
is considered. Our hypothesis, allowing $m(i)$ to vary with $i$,
permits to apply our result to a wider class of cases. An instance
of such a situation is presented in Example \ref{defin} below.

Our thesis only aims to compare $T_{k}$ $\ $\ with
$\widetilde{T}_{k}$, where $k$ is the highest level, which is
typically the one of specific interest.

The result in \cite{IG} on the contrary permits to compare $T_{h}$ \
with $\widetilde{T}_{h}$, for any $h=0,1, \ldots ,k$. However we can
be actually interested to prove $T_{k}$
$\preceq_{st}\widetilde{T}_{k}$, even in situations where
$\widetilde{T}_{h}\preceq_{st}T_{h}$ for some $h<k$.
\end{rem}
\begin{rem}
The method of proof of Theorem \ref{vedere} can be also convenient
for implementation in computer programs and it is based on the
skip-free property of Markov chains. A similar method of coupling
had been already implemented in \cite{FP05} with the aim to simulate
such Markov chains and to analyze Theorem 4.1 in \cite{IG}.

In our context, we also notice that such a method  leads to
efficient estimates of the difference between expected values of two
different hitting times. In fact, as can be easily proven, it is
more accurate that one based on the separate estimates of the two
expected values.   This circumstance turns out to be useful in
several situations of interest.

For instance, in the comparison between waiting times to words'
occurrences, where the expected values can be extremely large and
their differences relatively small, the estimate of expected values
might reveal numerically inaccurate if compared with direct
estimation of the difference between them.
\end{rem}

\begin{proof}
Let us fix a particular choice of $ m(1), \ldots , m(k-1)$ such that
i) and ii) hold. Moreover, for future convenience, we conventionally
fix $ m(k) =1 $. We will use a coupling method and we will obtain
the proof in a recursive way. \comment{Let us consider the space $
\Omega = [0,1]^{\mathbb{N}}$, $\mathcal{F}$ is the smallest
$\sigma$-algebra containing the Borel sets and the product Lebesgue
measure $\prod_{i =1}^\infty  Leb_i $.}

On a same probability space $(\Omega, \mathcal{F}, P)$, we  define a
sequence of i.i.d. random variables $\mathbf{U}= \{ U_n\}_{n \in
\mathbb{N} } $ and an independent array of i.i.d. random variables
$\mathbf{\widetilde U} =\{ \widetilde U_{k, n}\}_{k \in \mathbb{N},
n \in \mathbb{N}_+  } $. All these variables have uniform
distribution on $[0,1]$.

By using $ \mathbf{U} $, we will construct on $(\Omega, \mathcal{F},
P)$ a homogeneous Markov chain $ \mathbf{X} =(X_n)_{n\in
\mathbb{N}}$ having the law given by the initial distribution $ \pi
=(\pi_0, \ldots, \pi_k)$ and transition matrix $P= (p_{i,j})_{i, j
\in E_k} $. We will also construct,  by using $ \mathbf{U} $ and $
\mathbf{\widetilde U} $, a homogeneous Markov chain $
\mathbf{\widetilde X} =(\widetilde X_n)_{n \in \mathbb{N}}$ having
the law given by the initial distribution $ \widetilde \pi
=(\widetilde \pi_0, \ldots, \widetilde \pi_k)$ and transition matrix
$\widetilde P= (\widetilde p_{i,j})_{i, j \in E_k} $.
 We will
prove that the stopping times $T_k$  and $\widetilde T_k$,
corresponding to $\mathbf{X} $ and $ \mathbf{\widetilde X} $
respectively,  are ordered in the sense that
$$
 T_k(\omega) \leq \widetilde T_k(\omega),
$$
for each $\omega \in \Omega$.

First, we define $X_0$ and $\widetilde X_0$ with distribution $ \pi
$ and $\widetilde \pi$, respectively.

We set
\begin{equation}\label{tzero}
X_0 (U_0) := \inf\{ i \leq k: \sum_{l=0}^i \pi_l \geq U_0 \},
\end{equation}
and analogously
\begin{equation}\label{tzero2}
\widetilde X_0 (U_0) :=\inf\{ i \leq k: \sum_{l=0}^i \widetilde
\pi_l \geq U_0 \}.
\end{equation}
It is immediately seen that $ X_0 (U_0)\sim^d \pi  $ and $
\widetilde X_0 (U_0)\sim^d \widetilde \pi  $. Furthermore, for each
value $ u \in [0,1] $ $ X_0 (u) \geq \widetilde X_0 (u)$, in view of
the assumption $ \pi \succeq_{st} \widetilde{\pi} $.

Letting $ I(0)=I(0; U_0) := X_0 (U_0) $, and recalling the meaning
of
 $ m(1), \ldots , m(k-1),  m(k) =1 $,
we recursively define, for $n =1, 2, \ldots $
\begin{equation}\label{iterata1}
X_{ m(I(0))+ \ldots +  m(I(n-1)) }=X_{ m(I(0))+ \ldots +  m(I(n-1))
} (U_0, \ldots, U_n): =
  \inf\{ s \leq k: \sum_{l=0}^s p^{(m(I(n-1)))}_{ I(n-1),l} \geq U_{ n
  }\},
\end{equation}
\begin{equation}\label{iterata2}
 I(n)= I(n;U_0, \ldots, U_{n}) :=X_{ m(I(0))+ \ldots +  m(I(n-1)) }
 .
\end{equation}
 We notice that, if for a given  $n $, we obtain
$I(n) =k $ then also
$$
 I(n+1) =k   \hbox{ and  } X_{ m(I(0))+ \ldots +m(I(n-1))+1 } =k, \,\,\, a.s.
$$

We claim that
\begin{equation}\label{Tesplicito1}
  T_k  =  \sum_{r=1}^{L} m(I(r)),
\end{equation}
where $ L$ is the random index
\begin{equation}\label{indiceL}
    L : = \inf \{  l \in \mathbb{N}: I(l) = k  \} .
\end{equation}
We stipulate that  the sum (\ref{Tesplicito1}) is equal to zero if $
L=0 $. It is clear that,  when $ L=0$, (\ref{Tesplicito1}) holds.
Now we prove that also in the case $L >0$ the expression for $T_k$
given in (\ref{Tesplicito1}) holds true. In fact, at least the
inequality $T_k \leq \sum_{r=1}^{L} m(I(r))$ holds, since $I (L)=X_{
\sum_{r=1}^{L} m(I(r))} =k$, in view of positions (\ref{iterata1})
and (\ref{iterata2}).

We then want to show that for any $ t <  \sum_{r=1}^{L} m(I(r))$ one
has $X_t \neq k$. Let us first consider the values $a_s
=\sum_{r=1}^{s} m(I(r)) $ with $ s =1, \ldots , L-1 $. For these
values, $X_{a_s} =k $ would contradict the position (\ref{indiceL}).

Let us then consider the discrete intervals of the form $ B_s =
\{a_s +1, \ldots ,a_{s+1} -1\}$,
 with $s \in \{1, \ldots , L-1 \}$ and such that $ a_s +1 \leq  a_{s+1}
 -1$. For  $ a \in B_s$,  it is impossible that $X_a =k$. In fact  $
a+m(a ) \leq k$ for $ a \leq k-1$ and $ X_{c} - X_b \leq (c-b)$ for
any $c \geq b$.

\medskip

We now proceed to construct  $\widetilde T_k$. To this purpose we
consider a sequence of independent Markov chains $\{
\mathbf{\widetilde Y}^{(r)} \}_{r \in \mathbb{N}} $. For any $r=0
,1, \ldots$, the Markov chain $\mathbf{\widetilde Y}^{(r)}=
\{\widetilde Y^{(r)}_n \}_{n \in \mathbb{N}}$ will be such that
$\widetilde Y^{(r)}_0=0 $ with probability one and it will admit $
\widetilde P$ as transition matrix. More precisely
$\mathbf{\widetilde Y}^{(r)}$ is constructed in terms of $\{
\widetilde U_{r, 1} , \widetilde U_{r, 2}, \ldots \}$ as follows:
for $ n =1, 2, \ldots $
\begin{equation}\label{ytilde2}
 I^{(r)}(n-1) := \widetilde Y_{n-1}^{(r)}(\widetilde U_{r,1}, \ldots , \widetilde
U_{r,n-1}) ,
\end{equation}
\begin{equation}\label{ytilde1}
\widetilde Y_n^{(r)}(\widetilde U_{r,1}, \ldots , \widetilde
U_{r,n}) :=
  \inf\{ s \leq k: \sum_{l=0}^s \widetilde p_{ I^{(r)}(n-1),l} \geq \widetilde U_{r,  n
  }\}.
\end{equation}
As a function of $U_0 , U_1, , \ldots $, we now also define the
sequence $ \mathbf{\widetilde Y} = \{ \widetilde Y_n\}_{ n \in
\mathbb{N}}$ as follows:
\begin{equation}\label{iter}
 \widetilde Y_0=   \widetilde Y_0(U_0) := \widetilde X_0(U_0)
\end{equation}
\begin{equation}\label{iterata1a}
\widetilde Y_n(U_0, \ldots , U_n)  :=
  \inf\{ s \leq k: \sum_{l=0}^s \widetilde p^{(m(I(n-1)))}_{ I(n-1),l} \geq U_{ n
  }\}.
\end{equation}
Notice that the random variables $ I(0), I(1), \ldots $ appearing in
r.h.s. of (\ref{iterata1a}) have been defined in (\ref{iterata2}).
Furthermore we have, by construction, $ \widetilde Y_n (U_0, \ldots
, U_n) \leq I(n ;U_0, \ldots , U_n) $ in view of  condition ii). The
sequences $ \mathbf{\widetilde Y}$ and $\{ \mathbf{\widetilde
Y}^{(r)} \}_{r \in \mathbb{N}} $ are stochastically independent.

Let now, for $r \in \mathbb{N}$,
\begin{equation}\label{arresto1}
    N_1^{(r)} : = \inf \{ n\in \mathbb{N}: \widetilde Y_n^{(r)}  = \widetilde Y_r \},
\end{equation}
\begin{equation}\label{arresto2}
    N_2^{(r)} := \inf \{ n \in \mathbb{N}: \widetilde Y_n^{(r)}  =  I(r) \}.
\end{equation}
We notice that, for any $r =0, 1, 2, \ldots $, $ N_1^{(r)} \leq
N_2^{(r)} $ since the chain $ \mathbf{\widetilde Y}^{(r)} $ starts
in zero, it  increases at most of one unit at any step, and $
\widetilde Y_r \leq  I(r) $.  For any $r \in \mathbb{N}$, $
N_1^{(r)} $ and $ N_2^{(r)}$ are two stopping times with respect to
the filtration $( \mathcal{F}^{(r)}_n)_{ n \in \mathbb{N} }  $ where
$ \mathcal{F}^{(r)}_0  =\sigma (\widetilde Y_r, I(r))$ and
$$
\mathcal{F}^{(r)}_n = \sigma (\widetilde Y_r,   I(r), \widetilde
U_{r, 1}, \ldots , \widetilde U_{r, n}) \hbox{ for any } n=1,2,
\ldots .
$$
Now we consider the random variables
\begin{equation}\label{tempi1}
Z_r := \sum_{n=0}^r (N_2^{(n)}  - N_1^{(n)})+\sum_{n=0}^r m(I(n)),
\end{equation}
for $r =1, \ldots , L$ and
\begin{equation}\label{tempiale}
    \widetilde X_{Z_r} := \widetilde Y_r .
\end{equation}

Notice that, letting $r=L$ in \eqref{tempi1}, one has
\begin{equation}\label{identTk}
Z_L := \sum_{n=0}^L (N_2^{(n)}  - N_1^{(n)})+\sum_{n=0}^L m(I(n))=
\sum_{n=0}^L (N_2^{(n)}  - N_1^{(n)})+T_k,
\end{equation}
Furthermore, by recalling definition (\ref{arresto1}), $ \widetilde
X_{Z_r} =\widetilde Y_{N_1^{(r)}}^{(r)}$.

We now consider the following sequence of random variables:
\begin{equation}\label{sequen}
    \widetilde Y_0 ,  \widetilde Y_{N_1^{(0)}  +1  }^{(0)}, \ldots ,\widetilde
    Y_{N_2^{(0)}}^{(0)}, \widetilde Y_1 ,  \widetilde Y_{N_1^{(1)}   +1       }^{(1)}, \ldots ,\widetilde
    Y_{N_2^{(1)}}^{(1)}, \widetilde Y_2,
\end{equation}
obtained by gluing together the  sections of trajectories
$$ \widetilde
Y_0 ;  \widetilde Y_{N_1^{(0)}  +1  }^{(0)}, \ldots ,\widetilde
    Y_{N_2^{(0)}}^{(0)} ;  \widetilde Y_1;   \widetilde Y_{N_1^{(1)}
     +1       }^{(1)}, \ldots ,\widetilde
    Y_{N_2^{(1)}}^{(1)}; \widetilde Y_2 ; \ldots .
    $$

    Notice that some of the sections $ \widetilde Y_{N_1^{(r)}  +1  }^{(r)}, \ldots ,\widetilde
    Y_{N_2^{(r)}}^{(r)} $ can be missing. This happens when $ \widetilde Y_r = X_r
    $. We now set
\begin{equation}\label{tempiale2}
    \widetilde X_{Z_r+i} := \widetilde Y^{(r)}_{N_1^{(r)} +i}  , \,\,\, i
    =1, \ldots ,N_2^{(r)}-N_1^{(r)}.
\end{equation}
In view of the strong Markov property, the joint probability
distribution of  the random variables $ \widetilde X$'s defined by
(\ref{tempiale}) and (\ref{tempiale2}) coincides, by construction,
with a finite dimensional distribution for a Markov chain with
initial law $\widetilde \pi $ and transition matrix $ \widetilde P$.
The random variables $\widetilde X $'s have not been  defined for
any time $t \in \mathbb{N} $. By Kolmogorov's existence theorem,  we
can consider however the entire chain  $ \widetilde X_0 , \widetilde
X_1, \widetilde X_2, \ldots  $ by suitably adding variables at the
missing times. From \eqref{indiceL} and \eqref{identTk},we have
\begin{equation}\label{ffrr}
    \widetilde X_{ T_k  + \sum_{ r=0  }^{L} ( N_2^{(r)}  - N_1^{(r)})}
    =k.
\end{equation}
Furthermore, by repeating the same argument used above, we can also
obtain
\begin{equation}\label{ama}
    \widetilde X_l < k ,
\end{equation}
for $ l =0, \ldots ,{ T_k  + \sum_{ r=0  }^{L} ( N_2^{(r)}  -
N_1^{(r)})} -1$. Thus
\begin{equation}\label{precisa}
    \widetilde T_k = T_k  + \sum_{ r=0 }^{L} (
N_2^{(r)}  - N_1^{(r)}  ),
\end{equation}
 therefore $ \widetilde T_k \geq T_k  $, whence the stochastic
 comparison in (\ref{conclprop3})  follows.
\end{proof}

\begin{rem}\label{remogirone}
    The validity of the relation (\ref{precisa}), proven above, is much
more informative than simply $ \widetilde T_k \preceq_{st} T_k $,
and it allows in particular to provide different types of
inequalities.
\end{rem}

For our purposes we reformulate  Theorem 4.1 in \cite{IG} as
follows. Such a result gives a stronger conclusion with respect to
Theorem \ref{vedere} but under much stronger conditions.


\begin{theorem}\label{p1}
Let  $\mathbf{X}$ and $ \widetilde{\mathbf{X}}$ be skip-free Markov
chains with transition matrices   $ P=( p_{i,j})_{i,j \in E}$, $
\widetilde{P}=( \widetilde{p}_{i,j})_{i,j \in E}$, and initial
distributions  $\pi_0 =(\pi_0 (i))_{i \in E}$,  $\widetilde \pi_0
=(\widetilde \pi_0 (i))_{i \in E}$, respectively .
Under the conditions
    \begin{equation}\label{ordinatea}
    p_{i,\cdot } \succeq_{st} \widetilde{p}_{i,\cdot }  \hbox{ for each } i =0, \ldots,
    k-1,
\end{equation}
\begin{equation}\label{orro}
\pi_0 \succeq_{st} \widetilde \pi_0,
\end{equation}
 one has the stochastic comparison
\begin{equation}\label{th}
T_h \preceq_{st} \widetilde{T}_h , \hbox{ for } h =1,\ldots , k,
\end{equation}
 where $ T_h   $ is defined for $\mathbf{X}$
in (\ref{stop}) and $\widetilde{T}_h  $ is the analogue for $
\widetilde{\mathbf{X}}$.
\end{theorem}

We present an example in which the hypothesis of Theorem
\ref{vedere} are satisfied but both the hypothesis and the thesis of
Theorem \ref{p1} fail.

\begin{example}\label{defin}
    Let us consider $E = \{0,1,2,3\}$ and the transition matrices:
    \begin{equation}\label{exved}
P=\left(%
\begin{array}{cccc}
  \frac{1}{2}+\epsilon & \frac{1}{2}-\epsilon & 0 & 0 \\
  0 & \frac{1}{2} & \frac{1}{2} & 0 \\
  \frac{1}{2} & 0 & 0 & \frac{1}{2} \\
  0 & 0 & 0 & 1 \\
\end{array}%
\right), \,\,\,\,\,\,\, \widetilde P=
\left(%
\begin{array}{cccc}
  \frac{1}{2} & \frac{1}{2} & 0 & 0 \\
  1-\epsilon & 0 & \epsilon & 0 \\
  \frac{1}{2} & 0 & 0 & \frac{1}{2} \\
  0 & 0 & 0 & 1 \\
\end{array}%
\right),
\end{equation}
where $\epsilon \in (0, \frac{1}{2})$. We assume that the initial
measure of both Markov chains is concentrated on the state zero. For
no value  $ \epsilon \in (0, \frac{1}{2})$ hypothesis of Theorem
\ref{p1} is  satisfied. The first rows of the matrices $P^2 $ and $
\widetilde P^2 $ are respectively given by
 \begin{equation}\label{exved2}
p_{0, \cdot}^{(2)} =\left ((\frac{1}{2} +\epsilon)^2, \frac{1}{2} -
\epsilon^2- \frac{\epsilon}{2}, \frac{1}{4} - \frac{\epsilon}{2} , 0
\right ) , \,\,\,\,\, \widetilde{p}_{0, \cdot}^{(2)}=\left
(\frac{3-2 \epsilon}{4}, \frac{1}{4} , \frac{\epsilon}{2}, 0  \right
).
\end{equation}
Therefore we can take $ m(0)=2 $ and $m(1)=m(2)=1$ to verify the
hypothesis of Theorem~\ref{vedere}, when $\epsilon $ is small
enough. This shows that $ T_3 \preceq_{st} \widetilde T_3 $, but $
\widetilde{T}_1 \preceq_{st} T_1 $.
\end{example}

The following result appears, at a first glance, to be similar to
Theorem \ref{p1}. However it offers a much wider range of
applications. In  Section \ref{section3}, examples will be presented
in the frame of word occurrences. Also the proof of this result
 can be obtained along the
same line of Theorem \ref{vedere}, and will then be omitted.

\begin{theorem}\label{serve}
 Given two transition matrices in the space
$\Upsilon_{k}$, namely  $ P= (p_{i,j}: i, j =0, \ldots , k ) $ and $
\widetilde{P}= (\widetilde{p}_{i,j}: i, j =0, \ldots , k ) $. Let
the initial measures be  $ \pi = \widetilde{\pi} = \delta_0 $ (both
the Markov chains start in zero almost surely). Suppose that there
exists an integer $ m \in [1,k-1 ] $ such that
\begin{itemize}
    \item[(i)] $ \widetilde{T}_m \succeq_{st} T_m$,
    \item[(ii)] for each $ i \in [m, k-1]$ $
\widetilde{p}_{i,i+1} \leq p_{i,i+1} $ and $  \widetilde{p}_{i,0} +
\widetilde{p}_{i,i+1} =1 $.
\end{itemize}
Then $ \widetilde{T}_i \succeq_{st} {T}_i$ for $ i \in [m , k ] $.
\end{theorem}

\section{Analysis based on the asymptotic stochastic comparisons} \label{sec:asintotico}

In this section, we  consider the tail behaviors of hitting times $
T_k$ and $ \widetilde{T}_k$ corresponding to two Markov chains. More
precisely, we introduce in our analysis the asymptotic
stochastic comparison between $ T_k$ and $ \widetilde{T}_k$,  
according to the following definition.
\begin{definition}\label{asintord}
Given two random variables $X$ and $Y$ we write $ X \preceq_{a.st.}
Y$ if there exists $ t_0 \in \mathbb{R}$ such that $ P(X >t) \geq
P(Y >t)$, for each $ t \geq t_0 $.
\end{definition}
Consider the equivalence classes formed by probability distributions
over $\mathbb{R}$, that admit the same right tail. By considering
the quotient sets with respect to such equivalence relation the
asymptotic stochastic order is a partial order.

 Such a notion of ordering can find interesting
applications in probability (see also \cite{HaijunLi} and references
therein).

We point out that special conditions, such as the skip-free
property, are not needed in this section.

The following definition will be relevant to analyze the condition $
T_k \preceq_{a.st.} \widetilde{T}_k $. It  will moreover  provide
useful information to check the condition $ T_k \preceq_{st}
\widetilde{T}_k $.

\begin{definition}\label{rtriangolo}
Let $A=(a_{i,j}: i,j =0, \ldots , k)$ and $A'=(a'_{i,j}: i,j =0,
\ldots , k)$ be stochastic matrices. We write $A' \unlhd A $ if and
only if
\begin{equation}\label{bezecca}
    a'_{i, \cdot} \preceq_{st} a_{j, \cdot},  \,\,\,\,  \forall  i \leq
j \leq k .
\end{equation}
\end{definition}

\begin{rem}\label{remstochmon}
The relation  $\unlhd $ is stronger than \eqref{ordinatea}  and it
is transitive as well. However it is not reflexive. In this respect
we have that the relation $ P \unlhd P$ holds if and only if $P$ is
\emph{stochastically monotone}. More in general one can see that $A'
\unlhd A $ holds if and only if a stochastically monotone matrix $B
$ exists such that $ A' \unlhd B \unlhd A  $, see \cite{Sto} for a
general treatment of related topics.
\end{rem}

In view of our purposes, and for simplicity's sake, we consider in
 what follows  that the state $k $ is absorbing. This assumption however
is not restrictive at all.
  We also assume that the initial
measure for all the chains is concentrated on the state zero.

 The interest of Definition \ref{rtriangolo} in the present contest
 dwells in the next result.

\begin{theorem}\label{corord}
Let $\mathbf{X} =( X_n)_{n \in \mathbb{N}}$ and
$\mathbf{\widetilde{X}} =( \widetilde{X}_n)_{n \in \mathbb{N}}$ be
two Markov chains on $E_k =\{0, \ldots k\}$, both having initial
measure concentrated on zero, with transition matrices $ P$ and
$\widetilde{P}$ respectively. Let $k $ be an absorbing state for
both the matrices. Furthermore let $ \widetilde{P}^{n}\unlhd P^{n}$
for all $n $ large enough then $ T_k \preceq_{a.st.} \widetilde
T_k$.
\end{theorem}
\begin{proof}
We want to check that $ P( T_k > L) \leq  P( \widetilde T_k > L)$
for $L$ large enough. We remark, since  the state $k$ is absorbing,
that the  identity  $\{ T_k >L \} =\{ X_L \not = k \} $ holds. Then
$ P ( T_k >L ) = 1- p^{(L)}_{0, k }$ and similarly $ P ( \widetilde
T_k
>L ) = 1- \widetilde p^{(L)}_{0, k }$. In fact the
          condition $ \widetilde P^{L}\unlhd  P^{L}   $ in particular
          implies $\widetilde p^{(L)}_{0,k} \leq {p}^{(L)}_{0,k}$.
\end{proof}

We now present some results concerning the condition  $ P^{n}\unlhd
\widetilde P^{n}$ for $n $ large enough. First we give a
probabilistic characterization of the relation $\unlhd$.

\begin{lemma}\label{garibaldi}
Let $A$, $A'$ be stochastic matrices on the state space $E=\{0,
\ldots, k \}$.
 $ A \unlhd A' $ if and only if a Markov chain $(Z_n)_{n =0,1}$ with $ Z_n =(Y_n, Y'_n)$ on the state space $E^2$
 exists with the following properties:
 \begin{itemize}
    \item[i)] $ (Y_n)_{n =0,1} $ is a Markov chain with transition matrix
    $A$.  $(Y'_n)_{n =0,1}  $ is a Markov chain with transition matrix
    $A'$.
    \item[ii)] $P( Y_1 \leq Y_1' | Y_0=i , Y'_0=i' )=1$, for $i,i' \in E$, with $ i \leq i'
    $.
 \end{itemize}
\end{lemma}
\begin{proof}\label{corr}
Assume $A \unlhd A' $. Let us define two functions $ \phi, \phi'  :
[0,1 ] \to E $ as follows
\begin{equation}\label{genova}
  \phi
(u) :=\inf\{r \in E : \sum_{l=0}^r a_{i,l} \geq u \}, \,\,\, \phi'
(u) :=\inf\{ r \in E : \sum_{l=0}^r a'_{i',l} \geq u \}.
\end{equation}
We now consider the two random variables $ Y_1= \phi (U) $ and  $
Y'_1=\phi' (U) $, where $U$ is a uniform r.v. over $[0,1]$. 
Let $ Y_0$, $ Y'_0$ be  two $E$-valued random variables  such that
any pair in $E^2 $ is taken with positive probability by $
(Y_0,Y'_0)$ and independent of $U$.

 Then,
conditionally on $ Y_0=i $ (resp. $ Y'_0=i' $), $Y_1 $ (resp. $
Y'_1$) has the law $a_{i, \cdot}$ (resp. $ a'_{i, \cdot}$).
Furthermore ii) holds in view of \eqref{genova}.

Viceversa, if i) and ii) hold, then \eqref{bezecca} follows by
definition of stochastic ordering.
\end{proof}

The relation $ \unlhd$ is maintained under products of transition
matrices. We provide a direct proof based on probabilistic
arguments.

\begin{lemma}\label{mazzini}
If $A$, $A'$, $B$, $B'$ are stochastic matrices of  order $k$ such
that $ A \unlhd A' $ and $ B\unlhd B'$ then $ AB \unlhd A' B'$. If
$(A_i)_{i =1, \ldots , n }$ and $(B_i)_{i =1, \ldots , n }$ are
    stochastic matrices such that $ A_i \unlhd B_i$, for $i=1, \ldots, n
    $, then $ A_1 A_2 \ldots A_n   \unlhd  B_1 B_2 \ldots B_n  $.  

\end{lemma}
\begin{proof}\label{pr}
We want to construct two non-homogeneous Markov chains $ (Y_n)_{n
=0,1,2} $,  $ (Y'_n)_{n =0,1,2} $, with the following properties.
Let $ Y_0$, $ Y'_0$ be two $E$-valued random variables  such that
any pair in $E^2 $ is taken with positive probability by $
(Y_0,Y'_0)$.

Furthermore, the transition matrix
 of   $ (Y_n)_{n =0,1,2} $ is   $A$ for the first step and $B$
 for the second step. Analogously for $ (Y'_n)_{n =0,1,2}$ with $A'$ and $B'$.

On this purpose we consider two independent random variables  $U_1$
and $U_2$ uniformly distributed over $[0,1]$, also independent on
$(Y_0,Y'_0)$. Now we set, similarly to the proof of Lemma
\ref{garibaldi},
    \begin{equation}\label{genova2}
  Y_1
:=\inf\{r\in E: \sum_{l=0}^r a_{i,l} \geq U_1 \}, \,\,\, Y'_1
:=\inf\{ r \in E: \sum_{l=0}^r a'_{i',l} \geq U_1 \},
\end{equation}
 \begin{equation}\label{genova3}
  Y_2
 :=\inf\{r \in E: \sum_{l=0}^r b_{Y_1,l} \geq U_2 \}, \,\,\,
Y'_2  :=\inf\{ r \in E: \sum_{l=0}^r b'_{Y_1',l} \geq U_2 \}.
\end{equation}
Finally define $ X_0 =Y_0 $, $X_1 =Y_2$, $ X'_0 =Y'_0 $ and $ X'_1
=Y_2 $. The random variables  $ (X_n)_{n =0,1} $ form a Markov chain
with transition matrix
    $AB$; also $(X'_n)_{n =0,1}  $ is a Markov chain with transition matrix
    $A'B'$. The random pairs $(Z_n)_{n =0,1}$ defined by $
Z_n =(X_n, X'_n)$ can be seen as a Markov chain on the state space
$E^2$. In view of Lemma \ref{garibaldi}, the relation $ AB \unlhd A'
B'$ is proven by checking
 that $P( X_1 \leq X_1' | X_0=i , X'_0=i' )=1$, for $ i \leq i' \in E
 $. In this respect we have
\begin{equation}\label{primpas}
    P( X_1 \leq X_1' | X_0=i , X'_0=i' ) = P( Y_2 \leq Y_2' | Y_0=i ,
Y'_0=i' )=
\end{equation}
$$
=\sum_{ i_1 \leq i_1'} P(  Y_2 \leq Y_2' , Y_1 =i_1 ,Y'_1 =i'_1 |
Y_0=i , Y'_0=i' ).
$$
Notice that, in the last equality, we are allowed to reduce the sum
to $\{ (i_1 , i_1')\in E^2 : i_1 \leq i_1' \}$, in view of
\eqref{genova2}.

As to the r.h.s. of (\ref{primpas}) we obtain, by the Markov
property of $Z_n $,
\begin{equation}\label{quas}
\begin{array}{c}
  \sum_{ i_1 \leq i_1'} P(  Y_2 \leq Y_2' , Y_1 =i_1 ,Y'_1 =i'_1 |
Y_0=i , Y'_0=i' ) \\
  = \sum_{ i_1 \leq i_1'} P(  Y_2 \leq Y_2' | Y_1
=i_1 ,Y'_1 =i'_1 ) P(
  Y_1 =i_1 ,Y'_1 =i'_1 | Y_0=i , Y'_0=i'  ). \\
\end{array}
\end{equation}
We also have $ P(  Y_2 \leq Y_2' | Y_1 =i_1 ,Y'_1 =i'_1 ) =1 $, in
view of Lemma \ref{garibaldi}. Therefore the r.h.s. of \eqref{quas}
becomes $ P( Y_1  \leq Y'_1 | Y_0=i , Y'_0=i'  ) $.
 The latter term is equal to $1$ by \eqref{genova2} and this concludes the
 proof of the first part.
The second part is readily obtained from the first part by induction
and by taking into account associativity of the product of matrices.
\end{proof}

The next Theorem \ref{asintotica} has an immediate application to
our problem. It in fact provides an (apparently weaker) condition
sufficient for the hypothesis appearing in Theorem \ref{corord}.
Actually it gives an easily implementable condition to check the
asymptotic stochastic comparison between the hitting times for two
different Markov chains.

\begin{theorem}\label{asintotica}
 Let $\widetilde  P$ and $P$ be stochastic matrices of  order $k$. Assume that
 there exist two coprime integers $n_1$ and $n_2$ such that
 $$
\widetilde P^{n_1}\unlhd  P^{n_1} \hbox{ and }\widetilde
P^{n_2}\unlhd  P^{n_2},
 $$
 then
 \begin{equation}\label{tesi}
\widetilde P^{n}\unlhd   P^{n}
\end{equation}
for $n \geq  \hat n ( n_1, n_2) := inf \{ r: \forall  r' \geq r,
\,\,\,\,\,  r'= an_1 + b n_2 \hbox{ with } a,b \in \mathbb{N} \}$.
\end{theorem}
\begin{proof}\label{prasintotica}
For $ n \geq  \hat n ( n_1, n_2)  $ we can write, by definition of
$\hat n ( n_1, n_2) $, $ n = an_1 +b n_2$ for convenient natural
numbers $a$, $b$. Thus we can write $ P^n = (P^{n_1})^a  (P^{n_2})^b
$, $ \widetilde P^n = (\widetilde P^{n_1})^a  (\widetilde P^{n_2})^b
$. Then \eqref{tesi} is readily obtained by Lemma \ref{mazzini}.
\end{proof}
The number $ \hat n ( n_1, n_2)$ can be found using the Euclidean
algorithm. We notice that, in any case, $ \hat n ( n_1, n_2) \leq
n_1 n_2$.

 We already mentioned that the results of this section,
concerning the asymptotic stochastic ordering, can be used also to
check the usual stochastic ordering. A sufficient condition for $
T_k \preceq_{st} \widetilde{T}_k$ can be obtained as a simple
consequence of Theorem \ref{asintotica}.
\begin{proposition}\label{dueinsieme}
Let $\mathbf{X} =( X_n)_{n \in \mathbb{N}}$ and
$\mathbf{\widetilde{X}} =( \widetilde{X}_n)_{n \in \mathbb{N}}$ be
two Markov chains on $E_k =\{0, \ldots k\}$, both having initial
measure concentrated on zero, with transition matrices $ P$ and
$\widetilde{P}$ respectively. Let $k $ be an absorbing state for
both the matrices.

Assume that there exists two coprime $n_1$ and $n_2$ such that $
{\widetilde  P}^{n_1}\unlhd  P^{n_1} \hbox{ and } \widetilde
 P^{n_2}\unlhd
 P^{n_2}$, then
 \begin{itemize}
    \item[a.] $ T_k \preceq_{a.st.}
 \widetilde{T}_k$.
 \end{itemize}
 If, moreover, $ \widetilde{p}_{0,k}^{(n)} \leq  p_{0,k}^{(n)}$ for
 $n =1, \ldots , \hat n (n_1, n_2) -1$ then
 \begin{itemize}
    \item[b.] $ T_k \preceq_{st}
 \widetilde{T}_k$.
 \end{itemize}
\end{proposition}
\begin{proof}
The condition  $ T_k \preceq_{a.st.}
 \widetilde{T}_k$ is  equivalent to  $ \widetilde{p}^{(n)}_{0,k} \leq p^{(n)}_{0,k}$, for $n$ large enough,
(see also the  proof of Theorem \ref{corord}). We then obtain item
a. as a consequence of Theorem \ref{asintotica}. In fact the
          condition $ \widetilde P^{n}\unlhd  P^{n}   $ in particular
          implies $\widetilde p^{(n)}_{0,k} \leq {p}^{(n)}_{0,k}$.

          The condition $ T_k \preceq_{st} \widetilde{T}_k $ is
          equivalent to $ \widetilde{p}^{(n)}_{0,k} \leq p^{(n)}_{0,k}$, for $n =1,2, \ldots
          $.
          When the hypothesis of Theorem
          \ref{asintotica} holds, checking $T_k \preceq_{st} \widetilde{T}_k
          $ only requires that $ \widetilde{p}^{(n)}_{0,k} \leq {p}^{(n)}_{0,k}$,
          for $n =1,2, \ldots , \hat n (n_1, n_2) -1$, and then item b. is obtained.
\end{proof}
We notice that the conditions given in Theorem \ref{asintotica} can
be encountered rather often. A simple sufficient condition for its
hypothesis will be presented next. To this purpose we need the
following notation.

Given a stochastic matrix  $ P=(p_{i,j}: i,j =0, \ldots, k) $, such
that $ p_{i,k} < 1$ for $i =0, \ldots, k-1$, denote by $_{(k)} P =
(_{(k)} p_{i,j} : i,j =0, \ldots, k-1)$ the matrix obtained from $P$
by making $ k $ a taboo state: $_{(k)} p_{i,j} = p_{i,j}/(1-p_{i,k})
$. In view of our focusing on the hitting time in the state $k $, it
is not restrictive to assume  $ p_{k,k}=1 $.

Let $\mu = \mu (P) $ the modulus of the second eigenvalue of $P$ and
denote by $\lambda (P) $ the spectral gap of $P$ i.e.
\begin{equation}\label{spettri2}
\lambda (P) = 1 - \mu(P).
\end{equation}
Moreover, for a stochastic matrix $ P$, we use the term
\emph{ergodic} to designate the condition that there exists a
positive integer $ m $ such that all the elements of $ P^m $ are
strictly positive.

The eigenvalues of the transition matrix, admitting $k$ as an
absorbing state, determine the probability distribution of $T_k$,
under different types of conditions such as skip-free or
reversibility \cite{BrSh,DiaMic,Fill2,Fill1,Miclo,Zh}. Here we are
exclusively interested on the asymptotic stochastic ordering between
the hitting times for two different Markov chains. This restriction
allows us to focus attention on the second eigenvalue and to require
a mild condition (ergodicity of $_{(k)}P$), only.

\begin{theorem}\label{vicev}
    Let $ \widetilde{P}$ and $  P$ be two stochastic matrices on the
    state space
    $E =\{0 ,
\ldots , k\}$.
    Suppose that  $ \widetilde{p}_{k,k} =  p_{k,k} =1 $ and that
    $ _{(k)} \widetilde{P}$ 
    is ergodic. Assume
    furthermore that $\lambda (\widetilde{P}) < \lambda (  P)  $. Then
    there exists $n_0 $ such that $ \widetilde{P}^{n}\unlhd
    P^{n}$, for  $n \geq n_0 $.
\end{theorem}
\begin{proof} First we notice that $\lambda (  P)   $ is larger than zero.
This condition guarantees that, from each state $ i \in E $, the
Markov chain associated to $ P$ and starting in $0$ can reach the
state $k \in E$ in a finite number of steps, almost surely.
Actually, we will more precisely prove that there exists $  C>0 $
such that the following  inequality holds for any positive integer
$n$ and $i, j \in \{0 , \ldots , k-1\}$
\begin{equation}\label{uniformedis}
     p_{i, j}^{(n)} \leq  n^k C(1-\lambda ( P))^n .
\end{equation}
Then, by Borel-Cantelli Lemma,  the probability is zero of remaining
in $ E \setminus \{ k\}$ for an infinite number of steps.

 Concerning the matrix $\widetilde{P}$ we will prove, on the other hand,
that there exists $\widetilde{c}>0$ such that for $n $ large enough
\begin{equation}\label{uniformedis2}
    \widetilde{p}_{i, j}^{(n)} \geq \widetilde{c}(1- \lambda ( \widetilde{P}))^n  .
\end{equation}
If \eqref{uniformedis} and \eqref{uniformedis2} hold, we  get, for
$n$ large enough,  the inequalities
    $$
  \sum_{j=0}^{l} \widetilde{p}_{i, j}^{(n)} \geq  \sum_{j=0}^{l}  p_{\hat i, j}^{(n)}
    $$
    for any $l \in \{0, \ldots , k-1\}$ and any pairs $ i, \hat i \in \{0, \ldots , k-1\}$.
This, in particular, guarantees  $\widetilde{P}^{n}\unlhd P^{n}$ for
$n $ large enough and  concludes the proof.

   In order to get the
    inequality in \eqref{uniformedis} we can consider the Jordan
    representation
    $  P =  A^{-1} J  A  $
  for the stochastic matrix $ P$. In this representation
  we arrange in the increasing
  order the absolute values of eigenvalues on the main diagonal  and in particular $ ({J})_{k,k} =1$.

  Thus we can write $ |( J^n )_{i,j}|
  \leq n^k (1- \lambda ( P))^n$, for $ i,j  \in \{0, \ldots ,
  k-1\}$. By $  P^n =  A^{-1}  J^n  A $
  we obtain \eqref{uniformedis} in view of the assumption $ \lambda(  P) >0$.

In order to show \eqref{uniformedis2} we first notice that there is
only one eigenvalue
  of modulus $(1- \lambda ( \widetilde{P} ))$
as a consequence of the ergodicity of the transition matrix
$_{(k)}\widetilde{P}$. This is an easy consequence of
Perron-Frobenius theorem, see \cite{Sen}. We will denote by
$\widetilde{\mu}$ such an eigenvalue, which is a positive real
number (again as a consequence of  Perron-Frobenius theorem).

 We
start considering the case where $\lambda (\widetilde{P})
>0$. In such a
  case, similarly to above,  we use the Jordan representation of $\widetilde{P}$, i.e. $\widetilde{P}=
  \widetilde{A}^{-1} \widetilde{J} \widetilde{A}$.  We explicitly write
  \begin{equation}\label{sommat}
    \widetilde{p}_{i, j}^{(n)}= \sum_{l =0}^k \sum_{m =0}^k (\widetilde{A}^{-1}) _{i, l} (\widetilde{J}^n)_{l,m}
    (\widetilde{A})_{m,j} .
\end{equation}
To fix  ideas again we consider the case in which $
(\widetilde{J})_{k-1, k-1} = \widetilde{\mu} $ (the second highest
eigenvalue) and $(\widetilde{J})_{k,k}=1$. Furthermore, in
(\ref{sommat}), we are allowed to
limit attention only to indexes $0 \leq i \leq k-1 $. In fact, 
the terms $\widetilde{p}^{(n)}_{k,j}$ are zero for $j=0, \ldots ,
k-1$ and one for $j=k$.
From \eqref{sommat} we obtain
\begin{equation}\label{sommat2}
   \widetilde{p}_{i, k}^{(n)}=(\widetilde{A}^{-1})_{i,k} (\widetilde{A})_{k,k} + (\widetilde{A}^{-1}) _{i, k-1}
\widetilde{\mu}^n (\widetilde{A})_{k-1,k} + o( \widetilde{\mu}^n ),
\end{equation}
where $ (\widetilde{A}^{-1})_{i,k} (\widetilde{A})_{k,k}=1 $, for $i
=1, \ldots, k$, since the condition $ \lambda (\widetilde{P} )>0 $
guarantees $ \lim_{n \to \infty} \widetilde{p}_{i, k}^{(n)} =1$ and
$$
\lim_{n \to \infty} (\widetilde{A}^{-1}) _{i, k-1} \widetilde{\mu}^n
(\widetilde{A})_{k-1,k} + o( \widetilde{\mu}^n )=0 .
$$
We also obtain from (\ref{sommat}), for $ j =0, \ldots , k-1$,
\begin{equation}\label{sommat3}
   \widetilde{p}_{i, j}^{(n)}=(\widetilde{A}^{-1})_{i,k} (\widetilde{A})_{k,j} + (\widetilde{A}^{-1}) _{i, k-1}
\widetilde{\mu}^n (\widetilde{A})_{k-1,j} + o( \widetilde{\mu}^n ),
\end{equation}
but $ (\widetilde{A}^{-1})_{i,k} (\widetilde{A})_{k,j} =0 $ because
the limit $\lim_{n \to \infty} \widetilde{p}_{i, j}^{(n)} =0 $.

In this respect we claim that the products $
(\widetilde{A}^{-1})_{i, k-1} \widetilde{A}_{k-1, j} $ in
(\ref{sommat3}) can not be all equal to zero. In fact, if this were
the case, all the terms $ \widetilde{p}_{i, j}^{(n)} $ would not
depend on
$\widetilde{\mu} $ which is absurd. 
Therefore there exists $ \hat i , \hat j \in \{0, 1, \ldots , k-1\}$
such that $ (\widetilde{A}^{-1})_{\hat i, k-1} \widetilde{A}_{k-1,
\hat j } \not =0$. One can also deduce that $
(\widetilde{A}^{-1})_{\hat i, k-1} \widetilde{A}_{k-1, \hat j }  $
is a positive real number. In fact, from (\ref{sommat3}), one obtain
that
$$
\widetilde{p}_{\hat i, \hat j}^{(n)} = (\widetilde{A}^{-1})_{\hat i
, k-1} \widetilde{A}_{ k-1, \hat j } \widetilde{\mu} ^n + o
(\widetilde{\mu}^n ),
$$
with $\hat i , \hat j \in \{0, 1, \ldots , k-1\} $. Therefore
$(\widetilde{A}^{-1})_{\hat i , k-1} \widetilde{A}_{ k-1, \hat j }>0
$ because $ \widetilde{\mu}^n >0$ and $\widetilde{p}_{\hat i, \hat
j}^{(n)}>0 $. Furthermore, for $n $ large enough, $
\widetilde{p}_{\hat i, \hat j}^{(n)} > \hat c  \widetilde{\mu} ^n $,
where we let $\hat c : = \frac{1}{2} (\widetilde{A}^{-1})_{\hat i ,
k-1} \widetilde{A}_{ k-1, \hat j }$.

From the ergodicity of $_{(k) } \widetilde{P}$ we suppose that $n_0$
is a natural number such that the transition matrix $_{(k) }
\widetilde{P}^{n_0} $ has all the elements positive.

Taken a generic element $\widetilde{p}_{ i, j}^{(n)} $ with $ i,j
\in   \{0, 1, \ldots , k-1\}$ we obtain, for $n $ large enough,
\begin{equation}\label{termdomi}
 \widetilde{p}_{ i,  j}^{(n)}\geq \widetilde{p}_{ i,  \hat i}^{(n_0)}
 \widetilde{p}_{ \hat i,  \hat j}^{(n-2 n_0)}\widetilde{p}_{ \hat j ,
 j}^{(n_0)}\geq \widetilde{p}_{ i,  \hat i}^{(n_0)} \widetilde{p}_{ \hat j ,j}^{(n_0)}\hat c
 \widetilde{\mu}^{n-2n_0}
 \geq \left \{ \inf_{i,j \in \{ 0, \ldots, k-1\}} [ \widetilde{p}_{ i,  \hat i}^{(n_0)}
 \widetilde{p}_{ \hat j ,j}^{(n_0)} ] \hat c
 \widetilde{\mu}^{-2 n_0}\right \} \widetilde{\mu}^{n }= \widetilde{c}\,\, \widetilde{\mu}^{n },
\end{equation}
where  $ \widetilde{c} :=\inf_{i,j \in \{ 0, \ldots, k-1\}} [
\widetilde{p}_{ i, \hat i}^{(k_0)} \, \widetilde{p}_{ \hat j
,j}^{(k_0)} ] \hat c
 \widetilde{\mu}^{-2 k_0} $.

\medskip

We now consider the case  $\lambda(\widetilde{P}) =0$. This means
that the states $\{0, \ldots, k-1\}$ do not communicate with the
state $k$. Therefore there exists an invariant measure $ \pi
=(\pi_0, \pi_1, \ldots, \pi_{k-1}, 0) $ with $ \pi_i>0 $ for $ i =0,
\ldots, k-1$ (it is a consequence of the ergodicity of $
_{(k)}\widetilde{P} $ ). Therefore \eqref{uniformedis2} is trivially
satisfied. This ends the proof.

\end{proof}

The hypotheses of Theorem \ref{vicev} are met rather frequently and
then Proposition \ref{dueinsieme} provides an efficient tool to
check stochastic orderings between the hitting times of two Markov
chains on a same state space.

\begin{rem}\label{gd}
As a consequence of  Theorem \ref{vicev} we obtain, for a single
Markov chain  with  transition matrix $\widetilde{P}$ such that $
_{(k)}\widetilde{P} $ is ergodic, the large deviation equality $
\lim_{n \to \infty} \frac{1}{n}\ln (1-\widetilde{p}^{(n)}_{i,k}) =
\ln \widetilde{\mu} $ where $\widetilde{\mu} = 1 - \lambda
(\widetilde{P})$.
\end{rem}
\begin{rem}\label{gd22}
With obvious meaning of notation, consider the following conditions:
\begin{itemize}
    \item[a)] $ _{(k)}\widetilde{P} $ is ergodic and $\widetilde{\mu} <  \mu$ (as we
    have noticed in the proof of Theorem \ref{vicev} $\mu=\mu(P)$
    and $\widetilde\mu=\mu(\widetilde P)$ are positive real numbers when $ p_{k,k}=\widetilde  p_{k,k}=1  $ );
    \item[b)] There exists $  n_0 $ such that for all $ n >n_0 $   then $  \widetilde{P}^n
\unlhd  P^n$;
    \item[c)] $ T_k \preceq_{a.st.}  \widetilde{T}_k $.
\end{itemize}
By summarizing Theorem \ref{corord} and Theorem \ref{vicev}, we have
the implications $ a)\Rightarrow b) $ and $ b)\Rightarrow c) $.

We present two examples to show that the reverse implications fail.
First let us consider  the transition matrices
\begin{equation}\label{controesempi}
\widetilde{P}=  \left(%
\begin{array}{ccc}
  \alpha & (1-\alpha) & 0 \\
  0 & \beta &  (1-\beta) \\
  0 & 0 & 1 \\
\end{array}%
\right)   \hbox{ and }
 P = \left(%
\begin{array}{ccc}
   \beta &  (1-\beta) & 0\\
 0 & \alpha & (1-\alpha)  \\
  0 & 0 & 1 \\
\end{array}%
\right) ,
\end{equation}
with $ \alpha , \beta \in (0,1)$ and $ \alpha \neq \beta$. It is
immediately seen that both $ P$ and $\widetilde P $ have the
eigenvalues $\alpha$, $\beta$ and $1$. Moreover the hitting times $
T_2$ and $ \widetilde T_2$ have the same distribution. Therefore, in
particular, $ \widetilde{T}_2 \preceq_{a.st.}  T_2$ and $ {T}_2
\preceq_{a.st.}  \widetilde{T}_2$. In any case, for $n \in
\mathbb{N} $, the relations $ P^n \unlhd \widetilde P^n $ and $
\widetilde{P}^n \unlhd  P^n $ are not true. In fact
$\widetilde{p}_{0,0}^{(n)} = \alpha^n$,  $\widetilde{p}_{1,1}^{(n)}
= \beta^n$ while $p_{0,0}^{(n)} = \beta^n$, $p_{1,1}^{(n)} =
\alpha^n$.

We also notice that the two Markov chains, associated to the
transition matrices in (\ref{controesempi}), cannot be obtained one
from the other by simple permutations of the states. Therefore they
remain different also after changing the names of the states.

Thus we have obtained two different Markov chains with same
hitting-time distribution but such that the conditions  $
\widetilde{T}_2 \preceq_{a.st.}  T_2$ and $ {T}_2 \preceq_{a.st.}
\widetilde{T}_2$ cannot be detected by using the relation $ \unlhd$.
Whence  c) does not imply b).

\medskip

Secondly we show that item b) does not imply item a). Let us
consider the transition matrices
\begin{equation}\label{controesempi2}
\widetilde{P}=  \left(%
\begin{array}{ccc}
  \alpha & (1-\alpha) & 0 \\
  0 & \beta &  (1-\beta) \\
  0 & 0 & 1 \\
\end{array}%
\right)   \hbox{ and }
\hat P = \left(%
\begin{array}{ccc}
   \alpha &  (1-\alpha) & 0\\
 0 & \hat \beta & (1-\hat \beta)  \\
  0 & 0 & 1 \\
\end{array}%
\right) ,
\end{equation}
with $ \alpha , \beta , \hat \beta \in (0,1)$ and $\hat  \beta <
\beta < \alpha$. In both cases, the second largest eigenvalue is $
 \alpha $, therefore $ \widetilde{\mu} = \hat \mu = 1-\alpha $. In any case
 $  \widetilde{P} \unlhd \hat  P  $ and
therefore we obtain that $   \widetilde{P}^n \unlhd  \hat P^n$, for
any $n \in \mathbb{N}$. In particular we obtain that $  \hat T_2
\preceq_{a.st. }  \widetilde{T}_2 $. However item a), concerning the
two matrices $ \widetilde{P}$ and $ \hat{P}$, does not hold.
\end{rem}

We can thus conclude as follows: even if the relation $\unlhd $ may
appear very restrictive at a first glance, we notice however that it
provides a tool, to check $ T_k \preceq_{a.st.} \widetilde{T}_k $,
more frequently applicable than the comparison between the two
spectral gaps.

Moreover the use of $ \unlhd $ has the following advantages:
\begin{itemize}
    \item  As shown by Theorem \ref{asintotica} the use of the comparison $ \unlhd $  only requires the
    computation of powers of transition matrices and the latter operation is generally easier than computing
    eigenvalues.
    \item When the entries of the two transition matrices are all rational, the computation of powers
can be reduced to the case of integer entries, which allows one to
obtain easier and definite answers. We notice that simple cases of
matrices with rational entries are for example encountered in the
analysis of occurrences of words.
    \item The relation $ \unlhd $  may also be used to establish the condition
     $ T_k \preceq_{st} \widetilde{T}_k
          $ as shown by Proposition~\ref{dueinsieme}.
\end{itemize}
A few words about the comparison between the hypotheses of Theorem
\ref{vedere} and Theorem \ref{asintotica} are in order. The
condition $ \widetilde{P} \unlhd P $ implies (\ref{ordinatea}), in
Theorem \ref{p1}, which in turn is stronger than i) and ii) in
Theorem \ref{vedere}. However, as a consequence of Lemma
\ref{mazzini}, the hypothesis used in Theorem \ref{asintotica} is
much weaker than $ \widetilde{P} \unlhd P $.


\section{Comparisons for times of occurrences of words and leading numbers}\label{section3}

In this section we discuss some applications of the results of
Section \ref{sec2} and Section \ref{sec:asintotico}
in the frame of words occurrences. Let $\mathcal{A}_{N}\equiv
\{a_{1},\ldots ,a_{N}\}$ be the \textit{alphabet} composed by the
$N$ \textit{letters} $a_{1},\ldots,a_{N}$.\ \ An ordered sequence
$\mathbf{w}\equiv w_{1}w_{2}\ldots w_{k}$, where each of the
elements $w_{j}$ belongs to $\mathcal{A}_{N}$, is then seen as a
\textit{word} \textit{of length }$k$ \textit{on} $\mathcal{A}_{N}$.
We consider the space $\mathcal{A}_{N}^{k}$ of all possible words of
length $k$ on $\mathcal{A}_{N}$.

Assume that, at any instant $n=1,2,\ldots $, a letter is drawn at
random from  $\mathcal{A}_{N}$. Drawings are supposed to be
independent and
uniformly distributed over $\mathcal{A}_{N}$. We define the space $\Omega=\mathcal{A}_{N}%
^{\mathbb{N}}$; for
$\omega=(\omega_{1},\omega_{2},\ldots)\in\Omega$, we refer to
$\omega_{n}$ as \emph{the letter at time} $n\in\mathbb{N}$. The
probability measure on $\Omega$ is then the product measure that, at
any drawing, assigns probability $1/N$ to each letter: $
P(\omega_{n}=a)=\frac{1}{N},\,\,\,\,a\in\mathcal{A}_{N},\,\,\,n\in\mathbb{N}$.

 For any word $\mathbf{w}\equiv w_{1}w_{2}\ldots w_{k}$,
$\mathbf{w}\in \mathcal{A}_{N}^{k}$, we consider the stopping time
\[
T_{\mathbf{w}}:=\inf\{n\geq k|\omega_{n-k+1}=w_{1},\ldots
,\omega_{n}=w_{k}\},
\]
i.e. the random time until the first \emph{occurrence} of
$\mathbf{w}$.

This scheme gives also rise to an homogeneous Markov chain $\mathbf{X}%
=\{X_{n}\}_{n\in\mathbb{N}}$ with state space $E\equiv\{0,1,\ldots
,k\}$. The Markov chain $\mathbf{X}$ is defined as follows:

\begin{itemize}
\item[i)] $X_{0}=0$.

\item[ii)] For $n\geq1$ and $i\in\{1,\ldots,k\wedge n\}$, one has $X_{n}=i$ if

\begin{description}
\item[a.] $\omega_{n-i+1}\omega_{n-i+2}\ldots\omega_{n}=w_{1}w_{2}\ldots w_{i}
$

\item[b.] $\omega_{n-h+1}\omega_{n-h+2} \ldots\omega_{n}\neq w_{1}w_{2} \ldots
w_{h},\forall h=i+1,\ldots,k\wedge n; $
\end{description}

\item[iii)] One has $X_{n}=0$ if $\,\,\,\,\omega_{n-i+1}\omega_{n-i+2}%
\ldots\omega_{n}\neq w_{1}w_{2}\ldots w_{i}$ for all $1\leq i\leq
k\wedge n.$
\end{itemize}

\medskip
Under these positions, $T_{\mathbf{w}}$ coincides with the 
time $T_{k}$ of first visit to the state $k$ for the chain.
For our purposes we sometimes denote by $P^{(\mathbf{w})}=(p_{i,j}%
^{(\mathbf{w})})$ the transition matrix of such a Markov chain
associated to $\mathbf{w}$ $\in\mathcal{A}_{N}^{k}$ and denote by
$\mathcal{A}_{\mathbf{w} }
$ the alphabet formed by all the distinct letters belonging to
$\mathbf{w}$. The alphabet $\mathcal{A}_{\mathbf{w}}$ will be called
the \textit{minimal alphabet} of $\mathbf{w}$.

Furthermore, for $\mathbf{w}\equiv w_{1}w_{2}\ldots w_{k}$, we
denote by $\varepsilon_{\mathbf{w}}$ the \textit{leading number
}associated to\textit{ }$\mathbf{w}$. The latter is defined as the
binary vector
\begin{equation*}
    \varepsilon_{\mathbf{w}}:=(\varepsilon_{\mathbf{w}}\left(
1\right) ,\varepsilon_{\mathbf{w}}\left(  2\right)
,\ldots,\varepsilon_{\mathbf{w} }\left(  k\right)  )
\end{equation*}
where each $\varepsilon_{\mathbf{w}}\left(  u\right)  $ is equal to
$0$ or to $1$, according to the following position: for
$u=1,2,\ldots,k$
\begin{equation}\label{leadingnumber}
    \varepsilon_{\mathbf{w}}( u)
:=\mathbf{1}_{\{w_{k-u+1}=w_{1},\ldots,w_{k} =\ldots w_{u}\}}.
\end{equation}
Leading numbers have been introduced by J. Conway and have been
repeatedly used in the applied probability literature (see in
particular \cite{Li}, \cite{RD99} \cite{CZ1}), to deal with the
stochastic framework described above. In particular, the
distribution of $T_{\mathbf{w}}$ only depends on the leading number
$\varepsilon_{\mathbf{w}}$ and, as a function of it, the mean value
$\mathbb{E}(T_{\mathbf{w}})$ has the explicit expression $\mathbb{E}
(T_{\mathbf{w}})=\sum_{u=1}^{k}N^{u}\varepsilon_{\mathbf{w}}(u)$.

\bigskip

In what follows, we rather analyze stochastic comparisons between
the times $T_{\mathbf{w}}$ and $T_{\mathbf{w}^{\prime}}$ of
occurrences for two different words $\mathbf{w}$ and
$\mathbf{w}^{\prime}$ of the same length $k$. On this purpose we
apply the results of previous sections. We shall see furthermore
that an analysis based on the leading numbers $\varepsilon
_{\mathbf{w}}$ and $\varepsilon_{\mathbf{w}^{\prime}}$ can usefully
be combined with such results.

Let then $P_{\mathcal{A}}=P_{\mathcal{A}}^{(\mathbf{w})}$ and $P_{\mathcal{A}%
}^{\prime}=P_{\mathcal{A}}^{(\mathbf{w}^{\prime})}$ be the
transition matrices corresponding to the two words $\mathbf{w}$,
$\mathbf{w}^{\prime}$.

For several pairs $\mathbf{w}$, $\mathbf{w}^{\prime}$, it can happen
that $P_{\mathcal{A}}$ and $P_{\mathcal{A}}^{\prime}$\ satisfy a
condition of the type \eqref{ordinatea}. Theorem \ref{p1} then gives
us a useful criterion to check
$T_{\mathbf{w}}\preceq_{st}T_{\mathbf{w}^{\prime}}$.

The stochastic ordering $\preceq_{st}$ is a partial order on the
distributions of the times $T_{\mathbf{w}}$. As a first application
of Theorem \ref{p1} we now show that such a partial order does admit
a maximal element.

For $a \in\mathcal{A}_{N}$, let $\underline{\mathbf{a}}$ be the word
belonging to $\mathcal{A}_{N}^{k} $ and containing all letters equal
to $a$.

\begin{proposition}\label{pio2} For any word $\mathbf{w}\in\mathcal{A}%
_{N}^{k}$ and $\underline{\mathbf{a}}\in\mathcal{A}_{N}^{k}$, we
have $T_{\mathbf{w}} \preceq_{st} T_{\underline{\mathbf{a}}}$.
\end{proposition}
\begin{proof}
The transition probabilities for the chain associated to $\underline
{\mathbf{a}}$ are given by
\[
p_{i,i+1}^{(\underline{\mathbf{a}})}=\frac{1}{N},\hbox{ }p_{i,0}%
^{(\underline{\mathbf{a}})}=1-\frac{1}{N},
\]
for $0\leq i\leq k-1$. We then see that the proof is immediately
obtained from Theorem~\ref{p1}.
\end{proof}

\medskip

Under a simple condition, the following result shows that also a
minimal element does exist w.r.t. $\preceq_{st}$. For $N\geq k$, let
$\mathbf{\overline{w}}$ be the word $a_{1}a_{2}\ldots a_{k}$, made
with the first letters of the alphabet.

\begin{proposition}\label{p3} Let $N\geq k$. For any word $\mathbf{w}%
\in\mathcal{A}_{N}^{k}$ and for
$\mathbf{\overline{w}}\in\mathcal{A}_{N}^{k}$, we have
$T_{\mathbf{w}}\succeq_{st}T_{\mathbf{\overline{w}}}$.
\end{proposition}
\begin{proof}
If the leading number associated to the word
$\mathbf{w}=w_{1}w_{2}\ldots w_{k-1}w_{k}$ is $(0,0,\ldots,0,1)$,
then there is nothing to prove because the distributions of
$T_{\mathbf{w}}$ and $T_{\mathbf{\overline{w}}}$ are equal. Suppose
then that the leading number of the word $\mathbf{w}$ contains more
than only one $1$. In such a case $\mathbf{w}$ has a repetition of
at least one letter then $ \mathcal{A}_{\mathbf{w}}  $ is strictly
contained in $\mathcal{A}_N$, in view of the condition $N\geq k$.
Let for instance be $a_N \notin \mathcal{A}_{\mathbf{w}}$, say. Let
us consider the word $\mathbf{\widetilde{w}}=w_{1}w_{2}\ldots
w_{k-1}a_{N}$. It is clear that the leading number associated with
the word $\mathbf{\widetilde{w}}$ is
$(0,0,\ldots,0,1)$, see (\ref{leadingnumber}). Therefore the distribution of $T_{{\overline{\mathbf{w}}}%
}$ is the same as the one of $T_{{\widetilde{\mathbf{w}}}}.$ Hence,
in order
to show the stochastic comparison $T_{\mathbf{w}}\succeq_{st}%
T_{\mathbf{\overline{w}}}$, we prove $T_{\mathbf{w}}\succeq_{st}%
T_{\mathbf{\widetilde{w}}}$. The associated Markov chains are easy
to analyze because for $i=0,\ldots,k-2$ and $l=0,\ldots,k$ the
transition probabilities verify
\[
p_{i,l}^{(\mathbf{w})}=p_{i,l}^{(\mathbf{\widetilde{w}})}.
\]
As far as the transitions from the state $k-1$ are concerned, we
notice that for the index $\bar{l}=\max\{l=1,
\ldots,k-1:\varepsilon_{\mathbf{w}}(l)=1\}$ we can write $
p_{k-1,\bar{l}}^{(\mathbf{w})} =0$, since, by (\ref{leadingnumber}),
it must be $ w_{\bar{l}} =w_k $. On the other hand
$$
p_{k-1,\bar{l}}^{(\mathbf{\widetilde{w}}%
)}=\frac{1}{N}.
$$
In fact, when the Markov chain associated to $ \widetilde{w}$ is in
the state $k-1$, it has a transition toward the state $\bar{l}$ if
and only if  the letter $w_k$ is drawn. Moreover
\[
p_{k-1,l}^{(\mathbf{w})}=p_{k-1,l}^{(\mathbf{\widetilde{w}})},
\]
for $l\neq0,\bar{l}$. Therefore
\[
p_{k-1,0}^{(\mathbf{w})}=p_{k-1,0}^{(\mathbf{\widetilde{w}})}+\frac{1}{N}.
\]
We then see that the proof is immediately obtained from
Theorem~\ref{p1}.

\bigskip
\end{proof}

 \begin{rem}\label{cambioalfabeto} A same string $\mathbf{w}\equiv
w_{1}w_{2}\ldots w_{k}$ can be\ seen as a word on different
alphabets, and we must keep in mind which is the alphabet
$\mathcal{A}_{N}$ from which the random letters
$\omega_{1},\omega_{2},\ldots$ are drawn. Normally, such an alphabet
does not coincide with the minimal alphabet
$\mathcal{A}_{\mathbf{w}}$. The probability distribution of
$T_{\mathbf{w}}$ depends on $\mathbf{w}$ only through the leading
number $\varepsilon_{\mathbf{w}}$ and it depends on the alphabet
$\mathcal{A}_{N}$ only through its cardinality $N$. For brevity's
sake, such dependence on $N$ is omitted in our notation; however it
cannot be neglected, generally.
\end{rem}

In applying Theorem \ref{p1} to word occurrences, the following
proposition can be of interest.

\begin{proposition} \label{duestati} Let condition \eqref{ordinatea} hold for
$P_{\mathcal{A}} $ and $P_{\mathcal{A}}^{\prime}$. Then condition
\eqref{ordinatea} also holds for $P_{\hat{ \mathcal{A}} }$ and
$P_{\hat {\mathcal{A}}}^{\prime}$ for any alphabet $\hat{
\mathcal{A}} \supset \mathcal{A}_{\mathbf{w}}
\cup\mathcal{A}_{\mathbf{w}^{\prime}}$.
\end{proposition}
\begin{proof}
First we notice that \eqref{ordinatea} reads
\begin{equation}
\sum_{l=j}^{k}p_{i,l}\geq\sum_{l=j}^{k}p_{i,l}^{\prime}, \label{confr}%
\end{equation}
for $i=0,\ldots,k$ and $j=1,\ldots,k$.  Furthermore, for $j \geq 1$,
$p_{i,j}$, $p_{i,j}^{\prime}$ are equal to $1/N$ when are not null,
$N$ being the cardinality of ${\mathcal{A}}$.
 When the
alphabet ${\mathcal{A}}$ is replaced by the alphabet
$\hat{\mathcal{A}}$ then each $p_{i,l}$, with $l>0$, is replaced by
$p_{i,l}N/\hat{N}$ where $\hat{N}$ denotes the cardinality of
$\hat{\mathcal{A}}$. Then all the inequalities in \eqref{confr} are
maintained.
\end{proof}

\bigskip

Proposition \ref{duestati} guarantees the following property: once
we have proved the inequality
$T_{\mathbf{w}}\preceq_{st}T_{\mathbf{w}^{\prime}}$ by checking the
condition \eqref{ordinatea} for a sampling alphabet, then not only
$T_{\mathbf{w}}\preceq_{st}T_{\mathbf{w}^{\prime}}$ holds for any
other compatible alphabet, but also this conclusion stands still on
the comparison \eqref{ordinatea}.

\bigskip

In some cases Theorem~\ref{serve} can be used to compare two words
$\mathbf{w}$, $\widetilde{\mathbf{w}}$ which cannot be compared by
means of Theorem~\ref{p1}. An example follows.

\begin{example}
\label{exserve} Let $\mathbf{w}'=(A,A,B,A,A)$ and ${\mathbf{w}
}=(A,B,B,B,A)$ be seen as words on an alphabet with $N\geq5$. First
we notice that, for these two words, the hypothesis of Theorem
\ref{serve} holds true with $m=4$. On the other hand,
 by using Proposition \ref{p3}, one obtains
$T_{4}\preceq _{st}{T}'_{4}$. In fact the sub-word $ (A,B,B,B)$ made
with the first four letters of ${\mathbf{w} }$ has the leading
number $(0,0,0,1)$.

Furthermore condition (ii) of Theorem \ref{serve} is also satisfied.
Then we obtain that $T_{\mathbf{w}}\preceq_{st}T'_{ {\mathbf{w}}}$.
Notice that this example can be easily generalized by adding a same
number $\nu$ of letters $A$ on the left and on the right of the two
words and by adding a number $\mu$ of letters $B$ in the center. In
other terms we are saying that, by means of Theorem \ref{serve} and
Proposition \ref{p3}, one can compare two words $\mathbf{w}'$ and
${\mathbf{w}}$ whose leading numbers have the form
\[
\varepsilon_{{\mathbf{w}}}(i)=1,\hbox{ for }i=1,\ldots ,h;\hbox{ and
}\varepsilon_{{\mathbf{w}}} (i)=0,\hbox{ for }i=h+1,\ldots k-1,
\]
\[
\varepsilon_{\mathbf{w}'}(i)=1,\hbox{ for }i=1,\ldots ,h+1;\hbox{
and }\varepsilon_{\mathbf{w}'}(i)=0,\hbox{ for }i=h+2,\ldots k-1,
\]
with $N\geq k\geq2(h+1)$.
\end{example}

\comment{ 
\textbf{Remark 2} The circumstance that, for fixed cardinality of
the sampling alphabet, the probability distributions of
$T_{\mathbf{w}}$ only depends on $\varepsilon_{\mathbf{w}}$ has
important applications in the present framework. In order to check
$T_{\mathbf{w}}\preceq_{st}T_{\widetilde {\mathbf{w}}}$, we can look
for a different pair $\left(  \mathbf{w}^{\prime
},\widetilde{\mathbf{w}}^{\prime}\right)  $ such that $\varepsilon
_{\mathbf{w}}=\varepsilon_{\mathbf{w}^{\prime}}$,
$\varepsilon_{\widetilde
{\mathbf{w}}}=\varepsilon_{\widetilde{\mathbf{w}}^{\prime}}$, and
the
comparison $T_{\mathbf{w}^{\prime}}\preceq_{st}T_{\widetilde{\mathbf{w}%
}^{\prime}}$ can be directly established by using one of the results
presented above. This argument, combined with the above Remark 1,
also shows the interest of checking the possibility, for a given
word $\mathbf{w}$, that a different word $\mathbf{w}^{\prime}$
exists such that $\varepsilon
_{\mathbf{w}}=\varepsilon_{\mathbf{w}^{\prime}}$ and $\mathcal{A}%
_{\mathbf{w}^{\prime}}\subset\mathcal{A}_{\mathbf{w}}$.

\begin{proof}
\bigskip
\end{proof}

}

As noticed in  Remark \ref{remstochmon} of the previous section, the
condition $P' \trianglelefteq P$ between two stochastic matrices is
stronger than \eqref{ordinatea}, however in order to guarantee
 asymptotic comparisons,  we only need  $P'^n
\trianglelefteq P^n$, for all $n$ large enough.

Such a condition is not at all so strong. Actually we checked a
large number of pairs of words $ \mathbf{w}', \mathbf{w}$ with
$\mathbb{E}(T_{\mathbf{w}})< \mathbb{E} (T_{{\mathbf{w}'}})$ and
$_{(k)} {P'}$ ergodic, and in every case we found that $P'^n
\trianglelefteq P^n$, for all $n$ large enough. We notice that for a
transition matrix $P$ associated to a word $\mathbf{w}$, the
condition $_{(k)} {P'}$ ergodic just means that the state $0$ is
reachable from the state $k-1$. This condition certainly holds when
the alphabet $\mathcal{A}$ contains the minimal alphabet associated
to $\mathbf{w}$, strictly.


\begin{example}\label{parolefinale}
We used a simple program to compute powers of matrices and  compare
them in the sense $\unlhd$. In particular we compared  the
 matrices $P^{(\mathbf{w})}_\mathcal{A}$ and
$P^{(\mathbf{w'})}_\mathcal{A}$, for several pairs of words
$\mathbf{w}$ and $\mathbf{w'} $ on a same alphabet $ \mathcal{A}$.
Generally we have been able to check $ T_\mathbf{w} \preceq_{st}
T_\mathbf{w'} $ by finding $n_1$ and $n_2$ as in Theorem
\ref{asintotica} and by applying Proposition~\ref{dueinsieme}. For
instance for $ \mathbf{w}=(A,B,A,B,A)$ and $\mathbf{w'} =(
A,B,A,A,B) $, considered as two words on the binary alphabet, we
found  $ n_1 =3 $ and $n_2 =4 $. So that $\hat n =6$ and the
condition $ p_{0,k}^{(n)} \leq {p'}_{0,k}^{(n)}$ for $ n=1 , \ldots
, \hat n -1$  is immediately verified.

For $\mathbf{w}= (A,B,A,A,B)$ and $\mathbf{w'}=( A,B,A,B,A) $,
considered as words on the
 alphabet $\mathcal{A} =\{ A, B, C\}$, we found  $ n_1 =4 $ and $n_2 =5
 $. Also in this case the condition of Proposition~\ref{dueinsieme}
 is verified and $ T_\mathbf{w} \preceq_{st}
T_\mathbf{w'} $ is proven.
\end{example}


So far in this section we discussed the possibility to use the
concept of leading number in the analysis of the condition
$T_{{\mathbf{w}}} \preceq_{st}T_{{\mathbf{w}'}}$. More precisely, we
noticed that the leading number can be combined  with the results in
Section \ref{sec2}. We point out however that the same concept can
be usefully  combined also with the results, presented in Section
\ref{sec:asintotico}, based on comparisons of the form $P'^n
\trianglelefteq P^n$, for all $n$ large enough. In fact when the
above comparisons cannot be easily checked, one can try to find out,
on the same alphabet, a different pair of words $\mathbf{z}$,
$\mathbf{z}'$, with $ \varepsilon_\mathbf{z}=\varepsilon_\mathbf{w},
\,\,\,\, \varepsilon_\mathbf{z'}=\varepsilon_\mathbf{w'} $  so that
the results of Section \ref{sec:asintotico} become applicable. In
such a procedure we exploit, once again, the equivalence between the
conditions $ \varepsilon_\mathbf{z}=\varepsilon_\mathbf{w} $, and  $
T_\mathbf{z}$, $ T_\mathbf{w} $ are identically distributed.

\begin{example}\label{parolefinale2}
Let us compare the words, on the  alphabet $\mathcal{A} =\{A, B\}$,
$\mathbf{w}'=(A,B,B,B,A)$ and $\mathbf{w}=(A,A,A,A,B)$, for which
$\varepsilon_{\mathbf{w}'}=(1,0,0,0,1)$,
$\varepsilon_{\mathbf{w}}=(0,0,0,0,1)$ and
$\mathbb{E}(T_\mathbf{w'})=34$, $\mathbb{E}(T_{\mathbf{w}})=32$.

At least for $n\leq 30$, the powers of the corresponding transition
matrices are not comparable in the sense $\unlhd$.

Replace now $\mathbf{w}$ with $\mathbf{z}=(B,A,A,A,A)$. We have
again $\varepsilon_{\mathbf{z}}=(0,0,0,0,1)$. However, it is readily
seen that $ P_{\mathcal{A}}^{(\mathbf{w})} \unlhd
P_{\mathcal{A}}^{(\mathbf{z})} $, and we are in a position to even
apply Proposition \ref{dueinsieme}.
\end{example}

\bigskip\bigskip
\noindent
 {\bf Acknowledgments.}
We received, from an Associate Editor of EJP,  useful
bibliographical suggestions about recent literature on the
distributions of hitting times. Moshe Shaked recalled our attention
on the papers \cite {IG}, \cite {FP05}, and \cite {FP07b}.
 Haijun Li had let us see his paper \cite {HaijunLi} and
Antonio Di Crescenzo provided us with a useful remark.

\bibliographystyle{abbrv}



\end{document}